\documentclass[a4paper, 11pt,headings=small,abstract]{scrartcl}
\usepackage[english]{babel}

\usepackage{amscd,amsmath,amsthm,array,hhline,mathrsfs, enumerate, booktabs,fancybox,calc,textcomp,xcolor,graphicx,bbm,xspace,nicefrac,stmaryrd,url,arcs,listings,amssymb}
\usepackage[title]{appendix}
\usepackage[citestyle = numeric-comp,bibstyle = numeric, firstinits=true, natbib = true, backend = bibtex,maxbibnames=6, url=false, isbn=false, sortcites=true]{biblatex}
\usepackage[semibold]{libertine}
\usepackage[upint,amsthm]{libertinust1math}
\usepackage[scr=boondoxupr, scrscaled = 1.0, cal =stixtwoplain]{mathalpha}
\usepackage[scaled=0.95]{inconsolata}
\usepackage{accents}
\usepackage{dsfont}
\DeclareMathAlphabet{\mathsf}{OT1}{\sfdefault}{m}{n}
\SetMathAlphabet{\mathsf}{bold}{OT1}{\sfdefault}{b}{n}
\DeclareMathAlphabet{\mathfrak}{U}{jkpmia}{m}{it}
\SetMathAlphabet{\mathfrak}{bold}{U}{jkpmia}{bx}{it}
\usepackage[utf8]{inputenc}
\usepackage{enumitem}
\usepackage{etoolbox}
\usepackage[normalem]{ulem}
\usepackage{mathtools}
\usepackage{geometry}
\usepackage{float}

\usepackage{bm}
\usepackage{tikz}
\usepackage[outdir =./]{epstopdf}
\numberwithin{equation}{section}

\RequirePackage{algorithm}
\RequirePackage{algorithmic}

\usepackage{chngcntr}
\counterwithin{figure}{section}

\usepackage[dvipsnames]{xcolor}
\usepackage[colorlinks,linkcolor=RedViolet,citecolor=myteal,anchorcolor=Green,urlcolor=BlueViolet]{hyperref}

\definecolor{WIMgreen}{RGB}{60 134 132}
\definecolor{UMblue}{RGB}{4 47 86}
\definecolor{myteal}{RGB}{0 123 137}
\definecolor{material_green}{RGB}{27 43 52}
\definecolor{dracula_pink}{RGB}{180 93 149}
\definecolor{dracula_blue}{RGB}{40 42 54}
\definecolor{dracula_turq}{RGB}{92 143 159}
\definecolor{dracula_orange}{RGB}{255 184 108}
\definecolor{material_petrol}{RGB}{2 119 189}
\definecolor{Purple}{RGB}{103 58 183}

\newcommand{\e}{\mathrm{e}}
\usepackage{verbatim}
\usepackage[title]{appendix}

\definecolor{refkey}{gray}{.5}
\definecolor{labelkey}{gray}{.5}
\setkomafont{sectioning}{\normalcolor\bfseries}
\setkomafont{descriptionlabel}{\normalcolor\bfseries}
\setkomafont{section}{\large}
\setkomafont{subsection}{\normalsize}
\setkomafont{subsubsection}{\normalsize}
\setkomafont{paragraph}{\normalsize}
\setkomafont{subparagraph}{\normalsize}
\setkomafont{author}{\large}
\setkomafont{title}{\large\selectfont\bfseries}
\setkomafont{date}{\normalsize}

\usepackage{subcaption}

\theoremstyle{plain}
\newtheorem{theorem}{Theorem}
\newtheorem{lemma}[theorem]{Lemma}

\newtheorem{assumption}[theorem]{Assumption}
\newtheorem{corollary}[theorem]{Corollary}

\theoremstyle{definition}
\newtheorem{remark}[theorem]{Remark}
\newtheorem{example}[theorem]{Example}
\newtheorem*{theorem*}{Theorem}

\newtheorem{innercustomgeneric}{\customgenericname}
\providecommand{\customgenericname}{}
\newcommand{\newcustomtheorem}[2]{%
  \newenvironment{#1}[1]
  {%
   \renewcommand\customgenericname{#2}%
   \renewcommand\theinnercustomgeneric{##1}%
   \innercustomgeneric
  }
  {\endinnercustomgeneric}
}

\newcustomtheorem{customass}{Assumption}

\numberwithin{equation}{section}
\numberwithin{theorem}{section}

\DeclareMathOperator*{\argmin}{arg\,min}
\DeclareMathOperator*{\argmax}{arg\,max}

\DeclareMathOperator{\PP}{{\mathbb P}}
\DeclareMathOperator{\E}{{\mathbb E}}

\newcommand{\nn}{\mathbb N}

\newcommand{\qq}{\mathbb Q}
\newcommand{\rr}{\mathbb R}

\newcommand{\M}{\mathscr M}
\newcommand{\bb}{\cev{b}}

\newcommand{\C}{\operatorname c}
\newcommand{\auf}{[\![}
\newcommand{\zu}{]\!]}

\makeatletter
\DeclareRobustCommand{\cev}[1]{%
  \mathpalette\do@cev{#1}%
}
\newcommand{\do@cev}[2]{%
  \fix@cev{#1}{+}%
  \reflectbox{$\m@th#1\vec{\reflectbox{$\fix@cev{#1}{-}\m@th#1#2\fix@cev{#1}{+}$}}$}%
  \fix@cev{#1}{-}%
}
\newcommand{\fix@cev}[2]{%
  \ifx#1\displaystyle
    \mkern#23mu
  \else
    \ifx#1\textstyle
      \mkern#23mu
    \else
      \ifx#1\scriptstyle
        \mkern#22mu
      \else
        \mkern#22mu
      \fi
    \fi
  \fi
}

\makeatother

\numberwithin{equation}{section}
\numberwithin{theorem}{section}
\allowdisplaybreaks

\bibliography{artikel}

\title{\fontsize{16}{19} \selectfont Projection-based filtering in high dimensions via score-based diffusions}
\author{Sören Christensen\thanks{Kiel University, Department of Mathematics, Kiel, Germany\newline Email: \href{mailto:christensen@math.uni-kiel.de}{christensen@math.uni-kiel.de}} \and Jan Kallsen\thanks{Kiel University, Department of Mathematics, Kiel, Germany\newline Email: \href{mailto:kallsen@math.uni-kiel.de}{kallsen@math.uni-kiel.de}}  \and Claudia Strauch\thanks{Heidelberg University, Institute of Mathematics, Heidelberg, Germany\newline Email: \href{mailto:strauch@math.uni-heidelberg.de}{strauch@math.uni-heidelberg.de}} \and Lukas Trottner\thanks{University of Stuttgart, Institute for Stochastics and Applications, Stuttgart, Germany. \newline Email: \href{mailto:lukas.trottner@mathematik.uni-stuttgart.de}{lukas.trottner@misa.uni-stuttgart.de}}}
\date{\vspace{-15pt}}

\begin{document} 
\maketitle
\begin{abstract}
We consider the problem of recovering a latent signal $X$ from its noisy observation $Y$. The unknown law $\PP^X$ of $X$,
and in particular its support $\M$, are
accessible only through a large sample of i.i.d.\ observations.
We further assume
$\M$ to be a low-dimensional submanifold of a high-dimensional Euclidean space $\rr^d$. As a filter or denoiser $\widehat X$,
we suggest an estimator of the metric projection $\pi_\M(Y)$ of $Y$ onto the manifold $\M$.
To compute this estimator, we study an auxiliary semiparametric model in which $Y$ is obtained 
by adding isotropic Laplace noise to $X$. Using score matching within a corresponding diffusion model, we obtain an estimator of the Bayesian posterior $\PP^{X \mid Y}$
in this setup. Our main theoretical results show that, in the limit of high dimension $d$, this posterior $\PP^{X\mid Y}$
is concentrated near the desired metric projection $\pi_\M(Y)$.\\


\end{abstract}

\section{Introduction}
Recovering a latent signal $X$ from its noisy observation $Y$ is a fundamental and ubiquitous task in modern statistics,
central to fields such as nonparametric regression, denoising, filtering, and inverse problems.
We propose a data-driven strategy, based solely on an i.i.d.\
sample $x_1, \dots, x_n$ drawn from the marginal distribution $\PP^X$. We assume that the support $\M$ of $\PP^X$  is a subset of the high-dimensional Euclidean space $\rr^d$ with sufficiently small intrinsic dimension, which we encode in the form of a polarity hypothesis with respect to $d$-dimensional Brownian motion, see the next paragraph. Apart from that, no structural information about $\M$ beyond the sample is available,
as is common in modern applications where only high-dimensional unlabeled data are accessible.

From a Bayesian perspective, as in filtering, one could view the hidden $X$ as a random parameter whose  posterior law given $Y$ is to be determined.
This, however, would require assumptions on the kind of perturbation, that is, on the conditional law of $Y$ given $X$ (the likelihood), which is usually chosen according to a known but problem-specific data generating mechanism. The same applies to the maximum-likelihood approach in filtering, which leads to the Viterbi algorithm.

By contrast, our approach avoids any assumption on the precise dependence between $X$ and $Y$.
Instead, it is motivated by a geometric perspective: we seek the point on the support $\M$ of $\PP^X$ that lies closest to the observed $Y$.
Since $\M$ is unknown and represented only by the sample $x_1, \dots, x_n$, the statistical task can be viewed as one of manifold learning---namely, estimating the projection $\pi_\M(Y)$ of $Y$ onto $\M$ using only the empirical data.
In the absence of any information on the noise, this closest point in $\M$ is arguably the most natural guess about the latent signal $X$.

\begin{remark}
From a different viewpoint, our implicit assumption about the noise is that it makes sense to view the projection onto $\M$ as a reasonable filter. As an example, suppose that the observation $Y$ is linked to the latent signal $X$ via  $Y = X + \xi$ with some isotropic noise $\xi$ with density $f(x) = \varphi(\lvert x \rvert)$, where $\varphi$ is decreasing (e.g.\ centered Gaussian noise). The conditional likelihood in this model is thus given by $\varrho^{Y|X=x}(y) = \varphi(\lvert x - y\rvert)$. In a Bayesian modelling approach with known compact support $\M$ of $X$ (say a compact Lipschitz manifold of dimension $s \leq d-2$ which is covered by our setup below), the standard uninformative choice for the prior $\PP^X$ is the uniform distribution on $\M$ (with respect to the $s$-dimensional Hausdorff measure). Then, for a fixed observation $y$, the posterior density $x\mapsto\varrho^{X|Y=y}(x)$ is proportional to the likelihood $x\mapsto\varrho^{Y|X=x}(y)$. Since $\varphi$ is decreasing, the maximum-a-posteriori (MAP) estimator $\widehat{X}$ for the latent signal having observed $Y=y$ is therefore given by
\[\widehat{X} \in \argmax_{x \in \M} \varrho^{X|Y=y}(x) = \argmax_{x \in \M} \varrho^{Y|X=x}(y) = \argmax_{x \in \M} \varphi(\lvert x - y \rvert) = \pi_{\M}(y). \]
So, in the present setup, our objective can be interpreted as trying to maximise the a-posteriori probability in a Bayesian framework with unknown data support.
\end{remark}
\smallskip 

\noindent\textbf{Contribution of this paper} \quad
To explain our reconstruction strategy, we begin by formalising the problem. 
The objective is to denoise a perturbed input $y \in \rr^d$, interpreted as a realisation of a random variable $Y$.
This observation is assumed to be a noisy version of some unperturbed signal $x \in \rr^d$, drawn from a random variable $X$ with law $\alpha \coloneq  \PP^X$.
The distribution $\alpha$ represents the full population of clean data---for example, the theoretical law of digit images in the example discussed in Section~\ref{su:bilder}.

As our central and only assumption on the support $\M$ of $\alpha$ we suppose that it is almost surely never visited by a Brownian motion. More precisely, we assume the following for the remainder of the paper:
\begin{assumption}\label{ass:polar}
$\M = \operatorname{supp} \alpha$ is a polar set for Brownian motion $W$ on $\rr^d$, that is, 
\[\forall x \in \rr^d:\, \PP(\exists t > 0: x+W_t \in \M) = 0.\]
\end{assumption}
This class includes, for example, all Lipschitz submanifolds of codimension at least two (and thus also  at most countable unions thereof), see \cite[Section 2.4.2]{evans2018measure} and \cite[Chapter 8]{morters2010brownian}, which we adopt as a guiding intuition throughout the paper. Since the ambient dimension $d$ is assumed to be large, as commonly encountered in modern data settings such as image processing \cite{pope21}, our framework therefore also encompasses stratified data assumptions, which have been found to be geometrically appropriate for natural image data \cite{brown23} and token embeddings in LLMs \cite{li25,robinson26}, see \cite{martinez26} for a recent overview in the context of generative modelling.
While the distribution $\alpha$ is unknown, we assume that i.i.d.\ realisations $x_1, \dots, x_n$ of $X$ are available, 
that is, a sample of $n$ unperturbed images in our specific application. 
No  specific information is provided about how the raw data $X$ are perturbed to produce $Y$; 
in other words, the conditional law $\PP^{Y \mid X}$ is unknown, and no dataset of paired realisations $(X, Y)$, or at least of $Y$ alone, is assumed to be given.  We do assume, however, that the metric projection $\pi_{\M}(Y)$ of $Y$ onto the low-dimensional data support $\M$ is a suitable predictor of the latent signal $X$.
As mentioned above, we suggest to estimate the unperturbed $x$ corresponding to the observed noisy input $y$ by the closest element of the support $\M$ of $\alpha$, that is, by the minimiser of the distance $x\mapsto |x-y|$ from $y$. 
This can be viewed as a nearest-neighbour approach in which, however, the aim is to select $x$ 
from the ``true'' support $\M$ of all unperturbed images rather than from the limited set of observed samples $x_1, \dots, x_n$. 
Since $\M$ must be inferred from these data, this constitutes a non-trivial problem, which might appear even more challenging in high dimensions $d$. 
Surprisingly, however, we encounter a ``blessing'' rather than a curse of dimensionality in this context: a concentration-of-measure phenomenon ensures that our task can be addressed using a randomised algorithm whose output becomes almost deterministic in high dimensions.

To this end, we consider an auxiliary problem where the conditional law $\PP^{Y \mid X}$ is  explicitly given. More specifically, we assume that
\begin{equation}\label{e:noidea}
Y=X+\sigma\sqrt{U}V,
\end{equation}
where $X,U,V$ are independent, $U$ is exponentially distributed with  parameter 1, $V$ is standard Gaussian in $\rr^d$,
and $\sigma>0$ denotes a scaling parameter.
While the  model \eqref{e:noidea} may or may not be considered as interesting in its own right, the key observation is that in high dimensions, geometry dominates.
When $d$ is large, the posterior distribution $\PP^{X \mid Y = y}$ becomes highly concentrated in the region of $\M$ closest to the observation~$y$. In other words, sampling from this posterior can be viewed as an implementation of our above ``generalised nearest-neighbour problem''. Indeed, the following informal version of Theorem~\ref{t:nn}
states that the posterior concentrates around the (possibly set-valued) metric projection of $y$ onto the support $\M$ of $\PP^X$ at an exponential rate in $d$.
\begin{theorem*}
Let $\delta,\varepsilon>0$, and fix an observation $y\in\rr^d$.
If the prior $\alpha=\PP^X$ assigns at least mass $\varepsilon$ to some ball $B(y,r)$ with radius $r>0$ around $y$, then
\[\PP^{X \mid Y=y}\bigl(\M\cap B(y,(1+\delta)r)\bigr)\;\ge\; 1-\frac{1}{\varepsilon}(1+\delta)^{\,2-d}.\]
\end{theorem*}
In Section \ref{s:killed}, we present an algorithm that, given a realisation $y$ of $Y$,
aims at sampling from the posterior law $\PP^{X \mid Y=y}$ in the model \eqref{e:noidea}. It relies on interpreting the observation $Y$  as the terminal value of a Brownian motion started at $X$
and stopped at an exponentially distributed time.
In this setting, the posterior distribution $\PP^{X \mid Y=y}$ can be sampled efficiently using a time-homogeneous diffusion process with a stochastic terminal time,
whose drift coefficient is trained via score matching. Our analysis builds upon classical results for time-reversed Markov processes,
providing a rigorous study of score-based generative models with random time horizons.
\begin{remark}
The model \eqref{e:noidea} reflects the lack of specific information about the noise mechanism through the isotropy of the additive noise $\sigma \sqrt{U} V$.
One may wonder why a random factor $\sqrt{U}$ is introduced.
Note, however, that $|V|^2 / d = \tfrac{1}{d} \sum_{i=1}^d V_i^2 \approx 1$ for large~$d$ by the law of large numbers.
Hence, a non-deterministic factor is required in order to obtain a noise term with random magnitude.

The specific choice of an exponential distribution is motivated by its lack of memory, which allows us to interpret $Y$ as the terminal value of a killed Brownian motion and thus to regard the posterior through the lens of time-reversal of killed Markov processes.  Comparable denoising approaches through time-reversal of diffusions with additive Gaussian noise at a \textit{deterministic} terminal time \cite{ho2020denoising, song21}  have become immensely popular in recent years, but do not allow for a time-homogeneous structure of the denoising process. In contrast, this is enabled by our choice of Laplacian additive noise and forms the basis of our approach.

We want to stress that, regardless of whether the model \eqref{e:noidea} may be viewed as a reasonable suggestion to represent noise, it is merely used as a tool to numerically solve our projection problem which, by definition, does not rely on any probabilistic assumption on how the observation $Y$ is generated from the latent signal $X$.
\end{remark}

Our approach is inspired by recent advances in generative modelling, in particular diffusion-based methods that sample from complex distributions through iterative noise injection and denoising steps. 
Although such techniques originate from machine learning, we reinterpret them here as nonparametric estimators for conditional distributions. 
Our contribution lies in connecting these models to a high-dimensional statistical recovery problem under minimal structural assumptions, and in providing a rigorous justification of their validity via asymptotic concentration results.

We conclude this discussion with a brief outlook on how our approach can be applied in practice. Given training data $x_1, \dots, x_n$, a neural network
is trained to learn the drift function $\bb\colon\rr^d \to \rr^d$ of a time-reversed diffusion, see Remark~\ref{r:practice}.\ref{r:hier}.
The choice of the hyperparameter $\sigma$ in~\eqref{e:mc} should reflect the typical distance between the noisy observations $Y$ and the data manifold $\M$
supporting the unperturbed signals $X$, as described in the model formulation~\eqref{e:noidea}.
Given a new observation $y$, the corresponding latent $x$ is estimated by sampling a point $\widetilde x$ according to the backward diffusion process described in Remark~\ref{r:practice}.\ref{i:schwelle}. 
In sufficiently high dimensions, this random output~$\widetilde{x}$ will, with high probability, 
lie close to the nearest point on the manifold~$\M$, as stated in the theorem above. 
Even in moderate dimensions, the estimator~$\widetilde{x}$ admits a clear statistical interpretation: 
up to approximation errors due to imperfect learning of the drift~$\bb$ and discretisation of the backward process, the output~$\widetilde{x}$ constitutes a sample from the posterior distribution $\PP^{X \mid Y = y}$ in the parametric model~\eqref{e:noidea}. This interpretation, however, requires to consider the noise dynamics \eqref{e:noidea} to reflect reality, whereas it is only used as a numerical tool in our context.

In practice, the imperfect learning of $\bb$ prevents the simulated reverse diffusion from stopping exactly on $\M$. To account for this, we show in Section~\ref{s:high} that the high-dimensional concentration phenomenon is robust: provided that the estimated drift $\widetilde b_n$ is sufficiently accurate, an analogue of Theorem~\ref{t:nn} remains valid. Specifically, Theorem~\ref{t:nn2} establishes that the sampler concentrates, with high probability, in a small neighbourhood of $\M$ close to the metric projection of $y$.

To the best of our knowledge, the concentration of measure phenomena of score-based diffusion modelling in the limit of large ambient dimension have not yet been exploited or at least studied in the literature. Our methodology may therefore serve as a blueprint for related studies that rely on this effect.
\smallskip 

\noindent\textbf{Related work}\quad
Recovering a latent variable $X$ from noisy observations $Y$ is central to inverse problems and filtering. Classical Bayesian and filtering approaches typically assume a known or parametrically specified (possibly infinite-dimensional) conditional distribution $\PP^{Y \mid X}$, as in Kalman or particle filtering~\cite{doucet2009tutorial}, Bayesian inverse problems~\cite{dashti17, nickl23}, or deconvolution problems~\cite{meister09}. 
In contrast, nonparametric regression methods such as kernel and nearest-neighbour estimators infer regression functions or conditional expectations without explicitly specifying the data-generating mechanism~\cite{gyorfi2002distribution}, but they rely on paired samples $(X_i, Y_i)$. 
Our work shares this model-free spirit, yet addresses a fundamentally different information setting: we only observe i.i.d.\ samples from $\PP^X$ (unpaired) and a single $Y = y$.
Rather than projecting $y$ onto the finite training set, that is, the support of the empirical distribution $\PP_n^X \coloneq \tfrac{1}{n}\sum_{i=1}^n \delta_{X_i}$, we target the metric projection onto the (unknown) geometric support $\M = \operatorname{supp} \PP^X$. With regards to the Bayesian modelling approach with  unknown prior $\PP^X = \alpha$, our suggested algorithm in Remark \ref{r:practice} may be considered to be of ABC-type (Approximate Bayesian  Computation). However, it exploits the specific probabilistic structure of the model \eqref{e:noidea} and thus turns out to be different from generic nearest-neighbour ABC approaches, see Remark \ref{r:ABC} for details.

When the support of $\PP^X$ lies close to a low-dimensional submanifold $\M \subset \rr^d$, 
manifold learning methods such as Isomap~\cite{tenenbaum2000global}, LLE~\cite{roweis2000nonlinear}, and diffusion maps~\cite{coifman2006diffusion} can recover geometric information from data. 
Statistical guarantees for manifold reconstruction and local geometry estimation include, for instance, 
\cite{niyogi2008finding, genovese2012manifold, aamari19, aamari23}. 
These approaches are typically unsupervised and do not provide posterior inference given a single noisy observation~$Y$. 
In contrast, we do not attempt to estimate~$\M$ directly, but rather exploit its geometric structure implicitly to perform conditional inference for~$X$.

A complementary line of research investigates denoisers and score estimation.
Score matching \cite{hyvarinen2005estimation} and denoising autoencoders \cite{vincent11} show that denoising learns (an approximation to) the score,
i.e.\ the gradient of the log-density of a suitably smoothed distribution, see also \cite{alain2014regularized} for a refined analysis
of the reconstruction vector field as a score estimator.
This line of work demonstrates that learned denoisers can act as effective priors or regularisers in inverse problems, see
\cite{venkatakrishnan2013plug,chan2016plug,romano2017little}. However, the available theoretical guarantees are largely algorithmic:
they concern properties of the denoising operator (e.g.\ Lipschitz continuity or contractivity) or the convergence of iterative schemes
to fixed points of the denoising operator. Such results ensure stability of the procedure but do not yield distribution-level statements about the law of the recovered signal. 
In contrast, our analysis establishes \textit{global}, high-dimensional posterior concentration---at rates exponential in~$d$---onto metric projections, without requiring smoothness assumptions on the underlying data manifold.

Modern diffusion/score-based generative models \cite{ho2020denoising, song21} learn time-dependent scores to reverse a stochastic forward process, and have been extended to conditional sampling, inverse problems, and Bayesian inference; see \cite{ho21, dhariwal21, nichol22,  dou2024diffusion, denker24, janati2025mixture}. Training in these models typically assumes access to clean samples, which are then \textit{artificially corrupted} 
(e.g.\ by Brownian or Gaussian noise) in order to learn the reverse process.  
Our approach mirrors this pipeline only as a \textit{tractable prototype}:
we analyse isotropic multivariate Laplace corruption---realised as Brownian motion with exponential killing---to construct an efficient posterior sampler
with a random time horizon. Importantly, this auxiliary corruption is \textit{not} a modelling assumption on the true distribution~$\PP^{Y \mid X}$. 
Our main result on posterior concentration around the metric projection of~$Y$ onto~$\M$ 
suggests robustness with respect to the exact noise law. 
This, in turn, provides theoretical justification for projection-based denoising from \textit{unpaired} data. 
Related analyses that link denoising to local tangent information~\cite{stanczuk2024diffusion} 
are local in nature, whereas our result is global and measure-level.

Our approach should also be distinguished from the large body of work on guided diffusions for posterior sampling, see e.g.\ \cite{chung2023dps,wu2023tds,janati2025mixture,denker24,dou2024diffusion}.
There, the goal is to design sharp samplers for a posterior $\mathbb{P}^{X\mid Y=y}$, making use of a pretrained unconditional score for the prior $\mathbb{P}^{X}$. The likelihood $\mathbb{P}^{Y\mid X}$ is known and the reverse dynamics are tilted by a guidance term approximating the intractable gradient $\nabla_{x_t}\log p_t(y\mid x_t)$ of the likelihood of $y$ given the noisy state $x_t$. The initialisation is typically the same reference distribution as for unconditional generation. Problem-specific warm starts, where the reverse dynamics are started at an intermediate time from a noised version of the observation, are used as well \cite{meng2022sdedit,chung2022ccdf}. These allow shorter simulation times but require matching the observation carefully to a time point of the time-inhomogeneous reverse dynamics.
In our setting the situation is reversed: since the observation $y$ is interpreted as a state of the forward process, the posterior \eqref{e:hitting_distribution} is obtained from the \textit{unconditional} time-homogeneous backward diffusion \eqref{e:Z}. Conditioning enters only through the initial value $Y_0=y$ and no matching of $y$ to a noise level is necessary for efficient simulation. Moreover, since the auxiliary model \eqref{e:noidea} is merely a numerical device, we are not necessarily interested in the sampling accuracy for $\mathbb{P}^{X\mid Y=y}$  but focus on the denoising ability of the algorithm reflected in the concentration of this posterior around $\pi_{\M}(y)$.
Nonetheless, guidance can  be incorporated into our framework by fine-tuning the drift via a conditional $h$-transform \cite{christensen2025beyond}, which amounts to replacing the data distribution $\alpha$ by the conditional law of the data given some side information $u$, such as a class label or an indirect measurement. Started at $Y_0=y$, the guided process then concentrates near the metric projection of $y$ onto the support of this conditional law under the respective assumptions of Theorem~\ref{t:nn}. Developing training-free guidance for such random-horizon diffusions, where questions of manifold consistency as addressed in \cite{chung2022mcg,he2024mpgd} arise as well, is an interesting direction for future work.

From a probabilistic viewpoint, our framework builds on resolvent (Green) kernels and the properties of polar sets 
in classical potential theory~\cite{doob1984classical, blumenthal1968markov}. 
Technically, it relies on time-reversed Markov processes~\cite{chung2005markov} and on recent applications of time reversal to score-based models with random time horizons~\cite{christensen2025beyond}. Our main result also connects to high-dimensional probability
and the geometry of measures, where concentration-of-measure and geometric phenomena are central, see
\cite{ledoux2001concentration, buehlmann11, vershynin2018high, gorban18}. High-dimensional effects on nearest neighbour algorithms
are discussed in \cite{pestov13} and a general overview of such methods and their statistical properties can be found in  \cite{biau15}.

From a statistical perspective, our polarity assumption on the data support~$\M$ should be interpreted as a form of sparsity condition, commonly employed in high-dimensional statistics to characterise situations where the curse of dimensionality can be mitigated through structural properties of the data or of the underlying generative mechanism. 
Recent examples include nonparametric regression with deep neural networks under specific sparsity structures of the regression function~\cite{jsh20, kohler21, haya20}, and minimax-optimal convergence guarantees for deep generative models under sparsity assumptions on the data distribution, such as smooth submanifold supports 
or small-interaction models for factorisable densities; see~\cite{schreuder21, tangyang23, aamari24} for GANs and~\cite{azangulov24, tang24, fan25, kwon25} for diffusion models.

\smallskip 

\noindent\textbf{Structure of the paper}\quad
In Section~\ref{s:killed}, we introduce the auxiliary killed diffusion model and explain how to sample from the posterior distribution within this framework. 
Subsequently, in Section~\ref{s:convergence}, we investigate the high-dimensional limit, 
establishing the connection to our original objective of approximating the projection of the observation onto the unknown manifold of clean data. 
Finally, in Section~\ref{su:bilder}, we demonstrate how our general framework and theoretical results can be applied to an image processing setting. 
For completeness, the appendix recalls two auxiliary results of a technical nature.

\section{Interpretation as a killed diffusion model}\label{s:killed}
The key observation enabling posterior sampling in the parametric model~\eqref{e:noidea} is that the noising mechanism transforming a clean input $x$
into a perturbed observation $y$ can be realised by evolving a Brownian motion starting at $x$ until it is killed at an independent exponentially distributed time.
More precisely, we define a stochastic process $\bm{X} = (X_t)_{t \geq 0}$ in $\rr^d$, with $d\geq 3$, by
\[\PP^{X_0} = \alpha, \quad dX_t = dW_t, \quad \varrho_t = \frac{1}{\sigma^2},\]
where $W$ denotes a standard Brownian motion, $\varrho$ the killing rate, $\sigma > 0$ a fixed parameter, and $\alpha$ a probability distribution on $\rr^d$
supported on $\M$. As introduced in Assumption \ref{ass:polar}, $\M$ is assumed to be a polar set, 
that is, a set almost surely not hit after time~$0$ by Brownian motions with arbitrary initial law.
The process $\bm{X}$ is killed with constant killing rate $\varrho_t=1/\sigma^2$ at a random time $\tau$, i.e.\ $\tau$ is exponentially distributed
with parameter $1/\sigma^2$. We set $X_t\coloneq X_{\tau-}$ for $t\geq\tau$; that is, the process is actually stopped at the killing time $\tau$.

\begin{remark}
Scaling the Brownian motion by a constant factor $\widetilde\sigma > 0$ does not yield a more general model, as this is equivalent to replacing
$\sigma$ by $\sigma \widetilde\sigma$.
\end{remark}

Let $\bm{Y}=(Y_s)_{s\geq0}$ with $Y_s\coloneq X_{(\tau-s)\vee0}$ denote the time-reversed or backward process corresponding to $\bm{X}$.

\begin{remark}
It is straightforward to verify that $Y_0=X_\tau$ is related to $X_0$ as $Y$ to $X$ in \eqref{e:noidea}. More specifically, $Y_0 - X_0$ has the same distribution as~$\sigma \sqrt{U} V$, where~$U$ is exponentially distributed with parameter~$1$ and~$V$ is an independent standard Gaussian random variable in~$\rr^d$.
\end{remark}

In standard generative diffusion models, the noising process typically evolves until a fixed, deterministic terminal time~$\tau$. 
In contrast, our approach employs an independent random terminal time, specifically, an exponentially distributed killing time. While the general theoretical framework can be developed in close analogy to the classical case, two essential differences arise:
\begin{itemize}
\item The time-reversed process also evolves up to a random terminal time, which turns out to be a hitting time.
\item Unlike the standard setting, where the backward process is time-inhomogeneous, the reversed process here is a time-homogeneous diffusion.
\end{itemize}
This theoretical framework has recently been developed in \cite{christensen2025beyond}. For our specific model the main results can be summarised as follows.

\begin{corollary}[Backward diffusion]\label{theo:time_rev}
The time-reversed process $\bm{Y}$ is a diffusion relative to its own filtration, satisfying
\begin{align}\label{e:Z}
dY_s &= \bb(Y_s)\,ds + \bm{1}_{\{Y_s\notin\M\}}\, d\widetilde W_s\nonumber\\
&= \bb(Y_s)\bm{1}_{\{s<\zeta\}}\,ds + \bm{1}_{\{s<\zeta\}}\,d\widetilde W_s,
\end{align}
where $\widetilde W$ denotes a Brownian motion,
\begin{align}
\bb(y)&\coloneq
\begin{cases}
\nabla\log h(y) & \text{if }y\notin\M,\\
0 & \text{otherwise,}
\end{cases}
\label{e:b}\\
h(y)&\coloneq \int G_{\sigma^2}(x,y)\, \alpha(dx),\nonumber\\
G_{\sigma^2}(x,y)&\coloneq (2\pi)^{-d/2}\,\frac{2}{\sigma^2}\left(\frac{\sqrt{2}}{\sigma\,|x-y|}\right)^{\nu}
K_\nu\!\left(\frac{\sqrt{2}\,|x-y|}{\sigma}\right),\nonumber\\
\nu&\coloneq {d-2\over2},\nonumber\\
\zeta&\coloneq \inf\{s\in\rr_+: Y_s\in\M\},\label{e:sigma}
\end{align}
and $K_\nu$ denotes the modified Bessel function of the second kind.
In view of the Markovian structure, the same dynamics hold when conditioning on any initial value~$Y_0 = y\in\rr^d$, that is,
on the event that the forward process~$\bm{X}$ is killed at the point~$y = X_\tau$.
\end{corollary}
\begin{proof}
Indeed, in our setting the forward process $\bm{X}$ is a Brownian motion with exponential killing at rate $1/\sigma^2$. This process corresponds to a trivial $h$-transform (with $h \equiv 1$)
of standard Brownian motion, so its generator is simply the Laplacian with killing. The kernel $G_{\sigma^2}(x,y)$ appearing in \eqref{e:b}
is the Green function associated with Brownian motion killed at an exponential time with parameter $1/\sigma^2$, see \cite[p.~146]{christensen2019optimal}.
\cite[Proposition~5]{christensen2025beyond} therefore yields the SDE characterisation \eqref{e:Z} with drift \eqref{e:b} of the time reversal $\bm{Y}$.
Finally, since~$\M$ is polar for $\bm{X}$, it follows that its time reversal $\bm{Y}$ almost surely reaches~$\M$ only at its killing time, 
and is therefore naturally stopped at the first hitting time
$\zeta = \inf\{s \ge 0 : Y_s \in \M\}$,
see~\cite[Proposition~6]{christensen2025beyond}.
\end{proof}

\begin{remark}
The assumption that $\M$ is polar is crucial: it ensures that the time-reversed process $\bm{Y}$ reaches $\M$ only at its terminal time,
so that the stopping time $\zeta = \inf\{s \geq 0 : Y_s \in \M\}$ is a genuine hitting time. In particular, $\zeta$ is a stopping time relative to the filtration of $\bm{Y}$,
unlike~$\tau$ in the forward process, which is an external random time. Without the polarity assumption, the process could hit $\M$ with positive probability
before the terminal time, and the description of $\bm{Y}$ would require a more intricate stopping rule.
\end{remark}

Thanks to the Markov property of the backward diffusion $\bm{Y}$, its construction for arbitrary initial values $Y_0 = y$ yields the posterior law
$\PP^{X_0 \mid X_\tau = y}$ directly as the distribution $\PP^{Y_\zeta}$ of the hitting point $Y_\zeta$. This is summarised by the following theorem, which is a direct consequence of \cite[Proposition 5(3)]{christensen2025beyond}.

\begin{corollary}[Posterior law of $X_0$ given $X_\tau=y$]\label{t:hitting_distribution}
For Lebesgue-almost all $y\in\rr^d$,
\begin{equation}\label{e:bayes}
\PP^{X_0 \mid X_\tau=y}=\PP^{Y_\zeta \mid Y_0 = y},
\end{equation}
which is given by
\begin{align}\label{e:hitting_distribution}
\PP^{Y_\zeta \mid Y_0 = y}(dx)=\frac{G_{\sigma^2}(x,y)}{{h}(y)}\,\alpha(dx).
\end{align}
\end{corollary}

The following results provide the key for obtaining the drift field $\bb$ directly from data.
\begin{corollary}\label{coro:drift}
For $y \in \rr^d \setminus \M$, the drift function $\bb$ defined in~\eqref{e:b} satisfies
\begin{equation}\label{e:cond}
\bb(y) = \E\bigl(\nabla_2 \log G_{\sigma^2}(X_0, y)\ \big|\ X_\tau = y\bigr),
\end{equation}
where $\nabla_2$ denotes the gradient with respect to the second argument, i.e.\ differentiation in $y$.
\end{corollary}
\begin{proof}
By Corollary \ref{t:hitting_distribution}, for every $y \notin \mathscr{M}$,
\begin{align*}
\nabla \log h(y) = \int \frac{\nabla_2 G_{\sigma^2}(x,y)}{h(y)} \, \alpha(dx) &=   \E\bigl(\nabla_2 \log G_{\sigma^2}(Y_\zeta,y)\ \big|\ Y_0 = y\bigr)\\
&= \E\bigl(\nabla_2 \log G_{\sigma^2}(X_0,y)\ \big|\ X_\tau = y\bigr).\qedhere
\end{align*}
\end{proof}

Since the conditional expectation is an $L^2$-projection (see Lemma \ref{l:cond} for $B=A$, $C=\Omega$) and since
$\frac{d}{dz}\big(z^{-\nu} K_\nu(z)\big)=-z^{-\nu} K_{\nu+1}(z)$
(see for example \cite[Appendix 2]{Borodin}), \eqref{e:cond} suggests that $\bb$ from \eqref{e:b} minimises
\begin{align*}
b&\mapsto
\E\Bigl(\bigl| b(X_\tau) - \nabla_2 \log G_{\sigma^2}(X_0,X_\tau)\bigr|^2\Bigr)\\
&=\E\biggl(\biggl| b\bigl(X_0+\sigma\sqrt{U}V\bigr) + {\sqrt{2}\over\sigma}\frac{V}{|V|}\frac{K_{\nu+1}(\sqrt{2U}|V|)}{K_{\nu}(\sqrt{2U}|V|)}\biggr|^2\biggr)
\end{align*}
over all measurable functions $b\colon\rr^d\to\rr^d$, where $X_0,U,V$
denote independent random variables which are distributed according to $\alpha$, the exponential law with parameter 1,
and the standard normal distribution on $\rr^d$, respectively.
However, an asymptotic analysis of $\nabla_2 \log G_{\sigma^2}(X_0,X_\tau)$ suggests that the latter fails to be square-integrable. As a substitute, we obtain the
following result.

\begin{theorem}[Backward drift as minimiser of the denoising score
matching loss]\label{t:b}
Let $X_0, U, V$ be as above, and set for $\delta > 0$
\[
\M_\delta \coloneq \{ y \in \rr^d : d(y,\M) \le \delta \},
\quad\text{where}\quad d(y,\M) \coloneq \inf\{|x-y| : x \in \M\}.
\]
Then the function $\bb\colon\rr^d\to\rr^d$ defined in \eqref{e:b}
coincides on $\M_\delta^{\C}\coloneq \rr^d\setminus\M_\delta$
with the minimiser $\widetilde b$ of the objective
\begin{align}
b&\mapsto
\E\biggl(\biggl| b\bigl(X_0+\sigma\sqrt{U}V\bigr) + {\sqrt{2}\over\sigma}\frac{V}{|V|}\frac{K_{\nu+1}(\sqrt{2U}|V|)}{K_{\nu}(\sqrt{2U}|V|)}\biggr|^2\bm{1}_{\{\sigma\sqrt{U}|V|\geq\delta\}}\biggr)\label{e:MC}
\end{align}
over all measurable functions $b\colon\rr^d\to\rr^d$.
\end{theorem}

\begin{proof}
Straightforward calculations show that
$\bb(X_\tau) - \nabla_2 \log G_{\sigma^2}(X_0,X_\tau)$
has the same law as
\begin{align*}
& \bb\bigl(X_0+\sigma\sqrt{U}V\bigr)
- \nabla_2 \log G_{\sigma^2}\bigl(X_0,X_0+\sigma\sqrt{U}V\bigr)
= \bb\bigl(X_0+\sigma\sqrt{U}V\bigr) + {\sqrt{2}\over\sigma}\frac{V}{|V|}\frac{K_{\nu+1}(\sqrt{2U}|V|)}{K_{\nu}(\sqrt{2U}|V|)}.
\end{align*}
Since $\{X_0 + \sigma \sqrt{U} V \notin \M_\delta\} \subset \{ | \sigma \sqrt{U} V | \geq \delta\}$,
the claim follows now from the characterisation of the drift as a conditional expectation in Corollary \ref{coro:drift} and Lemma \ref{l:cond}.
\end{proof}

While the previous theorem provides a method for learning the drift field $b$ via score matching, the following result shows how the corresponding stopping rule for the backward process can be derived directly from it:

\begin{theorem}[Stopping rule]\label{t:stopping}
We have
\begin{align}\label{e:sigma2}
\zeta &= \inf\Bigl\{s\geq0: \sup_{r\leq s}|\bb(Y_r)|=\infty\Bigr\}\nonumber\\
&= \inf\Bigl\{s\geq0: \|\bb(\bm{Y})\|_{L^2(s)}=\infty\Bigr\}
\end{align}
for the stopping time in \eqref{e:sigma},
where we set
\[\| f\|_{L^2(s)}\coloneq\sqrt{\int_0^s|f(r)|^2dr}\]
for $f\colon\rr_+\to\rr^d$.
\end{theorem}

\begin{proof}
Set
\[\eta\coloneq \inf\Bigl\{s\ge0:\ \sup_{r\le s}|\bb(Y_r)|=\infty\Bigr\}\]
and
\[\xi\coloneq \inf\Bigl\{s\ge0:\ \|\bb(\bm{Y})\|_{L^2(s)}=\infty\Bigr\}.\]
We show that $\zeta\leq\eta\leq\xi\leq\zeta$.

\textit{Step 1:}
Since $\alpha$ is supported on $\M$, the integral $h(y) = \int G_{\sigma^2}(x,y)\,\alpha(dx)$ is finite and smooth on $\rr^d \setminus \M$,
so $\bb(y) = \nabla \log h(y)$ is locally bounded outside $\M$.

\textit{Step 2:}
If $s<\zeta$, then $Y_r\in \mathbb R^d\setminus\M$ for all $r\le s$, and $\bb$ is locally bounded there. Hence $\sup_{r\le s}|\bb(Y_r)|<\infty$, which implies $\zeta\leq\eta$.

\textit{Step 3:}
If $\int_0^s |\bb(Y_r)|^2dr$ is infinite for finite $s$, $\bb(\bm{Y})$ must be unbounded on $[0,s]$, which implies that $\eta\leq\xi$.

\textit{Step 4:}
We show by contradiction that $\xi$ is almost surely finite.
Otherwise, let $\xi_n \coloneq \inf\{s \geq 0: \|\bb(\bm{Y})\|_{L^2(s)} > n\}$.
By Novikov’s condition,
\[{d\qq\over d\PP} \coloneq \exp\biggl(-\int_0^{s\wedge \xi_n \wedge \zeta}\bb(Y_r)\, d\widetilde W_r
-{1\over2}\int_0^{s\wedge \xi_n \wedge \zeta}|\bb(Y_r)|^2\, dr\biggr)\]
defines a probability measure $\qq \sim \PP$ for any fixed $s\in\rr_+$.
Relative to $\qq$, the process $\bm{Y}$ has no drift and therefore behaves as a
standard Brownian motion on the stochastic interval $\auf0,s\wedge \xi_n \wedge \zeta\zu$.

Since $\PP(Y_\zeta \in \mathscr{M}) = 1$, we have $\PP(Y_\zeta \in \mathscr{M}, \xi = \infty) = \PP(\xi=\infty) > 0$.
As $\zeta < \infty$ almost surely, it follows that $\PP(Y_{s\wedge \zeta} \in \mathscr{M},\xi = \infty) > 0$ for some sufficiently large $s\in\rr_+$.
From $\xi_n \to \infty$ on $\{\xi = \infty\}$, we deduce that
\[\PP(Y_{s \wedge \xi_n \wedge \zeta} \in \mathscr{M}) \geq  \PP(Y_{s \wedge \xi_n \wedge \zeta} \in \mathscr{M}, \xi = \infty)
\underset{n \to \infty}{\longrightarrow} \PP(Y_{s \wedge \zeta} \in \mathscr{M}, \xi = \infty) > 0.\]
Thus, $\PP(Y_{s\wedge \xi_n \wedge \zeta} \in \mathscr{M}) > 0$ for some sufficiently large $n$.
However, $\qq(Y_{s\wedge \xi_n \wedge \zeta} \in \mathscr{M}) = 0$, because $\bm{Y}$ is a $\qq$-Brownian motion on $\auf0,s\wedge \xi_n \wedge \zeta\zu$
and $\mathscr{M}$ is polar for Brownian motion. This contradicts $\PP \sim \qq$, hence our initial assumption
$\PP(\xi = \infty) > 0$ cannot hold.

\textit{Step 5:}
We now show by contradiction that $\xi\leq\zeta$. Otherwise,
$\|\bb(\bm{Y})\|_{L^2(\zeta)}<\infty$ with positive probability.
Since $\|\bb(\bm{Y})\|_{L^2(s)}=\|\bb(\bm{Y})\|_{L^2(s\wedge\zeta)}\leq\|\bb(\bm{Y})\|_{L^2(\zeta)}$ for all $s\in\rr_+$, this would imply
$\xi=\infty$ with positive probability, contradicting Step~4.
\end{proof}

\begin{remark}[Implementation]\label{r:practice}
The identity \eqref{e:bayes} in Corollary \ref{t:hitting_distribution} implies that if we interpret our observation $y$ as a perturbed version of the latent variable $X_0$
(with prior distribution $\alpha$), and if this perturbation is modelled by $X_\tau$, then the posterior law of $X_0$
having observed $X_\tau=y$ coincides with the law of $Y_\zeta$. The following two remarks indicate how this can be used in practice, see also first experimental evidence provided in \cite{christensen2025beyond}.
\begin{enumerate}
\item\label{r:hier} The drift function $\bb$ can be learned by training a neural network on an empirical approximation of the $L^2$-contrast function given on the right-hand side of \eqref{e:MC}, which we denote by $L_b$. Indeed, the latter equals
\begin{equation}\label{e:mc}
\begin{split}
L_b &= \int_{\M}\iint  \biggl| b\bigl(x+\sigma\sqrt{u}v\bigr) + {\sqrt{2}\over\sigma}\frac{v}{|v|}\frac{K_{\nu+1}(\sqrt{2u}|v|)}{K_{\nu}(\sqrt{2u}|v|)}\biggr|^2\bm{1}_{\{\sigma\sqrt{u}|v|\geq\delta\}} \PP_U(du)\, \PP_V(dv) \, \alpha(dx) \\
&=\lim_{N\to\infty}{1\over N}\sum_{i=1}^N\biggl(\biggl| b\bigl(\xi_i+\sigma\sqrt{u_i}v_i\bigr) + {\sqrt{2}\over\sigma}
\frac{v_i}{|v_i|}\frac{K_{\nu+1}(\sqrt{2u_i}|v_i|)}{K_{\nu}(\sqrt{2u_i}|v_i|)}\biggr|^2\bm{1}_{\{\sigma\sqrt{u_i}|v_i|\geq\delta\}}\biggr)\\ 
&\eqcolon \lim_{N \to \infty}\widetilde{L}_{b,N}
\end{split}
\end{equation}
where $(\xi_i,u_i,v_i), i=1,\dots,N$ denote independent realisations distributed according to the law of $(X_0,U,V)$.
Note, however, that the law $\alpha$ of $X_0$ is typically unknown. As a substitute, the realisations $\xi_1,\dots,\xi_N$
can be sampled by bootstrapping (i.e.\ drawing with replacement) from the training data $x_1,\dots,x_n$ whose law equals $\alpha$. This procedure can potentially be further improved by using importance sampling techniques to ensure that sufficient training data has been generated at small noise scales, where accurate score estimation is fundamentally important to predict the score explosion termination criterion well.  

For finite $N$, a score approximation $\widetilde{b}_N$ is then obtained as an ex- or implicitly penalised empirical risk minimiser of $b \mapsto \widetilde{L}_{b,N}$, typically using neural networks. The empirical data distribution $\PP_n = \frac{1}{n} \sum_{i=1}^n \delta_{x_i}$ appears implicitly in the empirical contrast function $\widetilde{L}_{b,N}$ used to learn the approximation of $\bb = \nabla \log h$. Since in the joint model \eqref{e:noidea} the observation $Y$ has density  $\varrho^Y(y)= \tfrac{1}{\sigma^2} \int_{\M} G_{\sigma^2}(x,y)\, \alpha(dx) = \tfrac{1}{\sigma^2} h(y)$, the drift $\bb$ is equal to the score $\nabla \log \varrho^Y$ of the corrupted observation $Y$. But note that the marginal score is not estimated via an explicit plug-in approach yielding the log gradient of $ h_N(y) \coloneq \int G_{\sigma^2}(x,y)\, \PP_n(dx) \approx h(y)$, which would correspond to overfitting resulting from minimising $\widetilde{L}_{b,N}$ without any penalisation over all measurable $b$. Instead, an appropriate penalisation or restriction of the approximation class for the empirical denoising score matching objective properly smoothes the score estimator and avoids overfitting.
\item\label{i:schwelle} Once $\bb$ has been learned with an approximation $\widetilde{b}_N$, approximate samples $\widetilde x$ from $\PP^{X_0 \mid X_\tau=y}$ can be obtained by simulating a process $\widetilde{\bm{Y}}$ with $\widetilde{Y}_0=y$ and dynamics as in \eqref{e:Z} with $\bb$ replaced by $\widetilde{b}_N$
and stopping when $\|\widetilde{b}_N(\widetilde{\bm{Y}})\|_{L^2(s)}$ exceeds a large threshold, in accordance with \eqref{e:sigma2}. Robustness of the projection properties of $\widetilde{\bm{Y}}$ with respect to the approximation quality of $\widetilde{b}_N$ is analysed in Theorem \ref{t:nn2}. Here we emphasise that Theorem \ref{t:nn2} \textit{assumes} the existence of an estimator $\widetilde{b}_N$ with a certain $L^2$-estimation accuracy. The verification of such a non-asymptotic assumption with explicit specification  of precision parameters and neural network architectures in terms of the sample size $n$, the (intrinsic) dimensionality of the data and appropriate smoothness assumptions  requires significant additional statistical effort and remains open for our  random-horizon diffusion model. For standard diffusion models several works have been devoted to this research question in recent years, cf., e.g., \cite{azangulov24, oko23, kwon25, tang24}. Furthermore, we do not discuss numerical errors arising from the reverse SDE discretisation, which for standard diffusion models is a  very active research area as well, see \cite{arsenyan25,benton24,chen_chewi23,chen_lee23,conforti25,silveri2025}.
\end{enumerate}
\end{remark}
\begin{remark}[Generative modelling of $\bm{\alpha}$]\label{r:g}
For a fixed dimension $d$, as considered in this section, the model introduced above could also be used for unconditional sampling
from the prior distribution $\alpha$, similar to standard diffusion models. Indeed, consider $y \coloneq  y_m \coloneq  ma$ for some $a \in \rr^d$ with $|a| = 1$. Using that
\[\frac{G_{\sigma^2}(x,y)}{G_{\sigma^2}(0,y)}
\ \sim\
\exp\!\left(\frac{\sqrt{2}}{\sigma}\, a^\top x\right),
\qquad |y| = m \to \infty,\ \ y = m a,\ \ |a| = 1\]
uniformly in $x$ on compacts, see \cite[Proposition 3.2]{christensen2019optimal}, we obtain
\[
\frac{G_{\sigma^2}(x,y)}{h(y)}= \frac{G_{\sigma^2}(x,y)}{\int G_{\sigma^2}(x,z)\,\alpha(dz)}
= \frac{\frac{G_{\sigma^2}(x,y)}{G_{\sigma^2}(x,0)}}{\int \frac{G_{\sigma^2}(x,z)}{G_{\sigma^2}(x,0)}\, \alpha(dz)}
\approx \frac{\exp(\sqrt{2}/\sigma\, a^\top y)}{\int \exp(\sqrt{2}/\sigma\, a^\top z)\,\alpha(dz)} \to 1
\]
as $m \to \infty$ and subsequently $\sigma^2 \to \infty$.
Thus, for large initial values $y$ and large $\sigma$, it follows from \eqref{e:hitting_distribution} that
\[
\PP^{Y_\zeta \mid Y_0 = y}(dx) = \frac{G_{\sigma^2}(x,y)}{h(y)} \,\alpha(dx) \approx \alpha(dx),
\]
indicating that the process approximately samples from the data distribution $\alpha$. This shows that, beyond posterior inference,
the model can also serve as an unconditional generator, and the above procedure may be viewed as a random-time-horizon variant of
\textit{variance exploding} denoising diffusion models \cite{song21}. However, for the remainder of this paper, the focus will be on sampling
from the posterior distribution $\PP^{X_0|X_\tau=y}$ in a high-dimensional regime, considering moderate values of $|y|$.
\end{remark}

\begin{remark}[Interpretation as ABC algorithm]\label{r:ABC} There is a rich literature on Approximate Bayesian Computation (ABC) methods, cf.\ \cite{marin12,robert16} for survey treatments, which are primarily designed for posterior inference when the joint density $\varrho^{X,Y}$ of $\PP^{X,Y}$ (here $\PP^X$ being the prior and $\PP^{Y \mid X}$ the likelihood) is not available -- typically due to the likelihood being too expensive to evaluate numerically -- thus making MCMC sampling approaches based on $\varrho^{X,Y}$ useless. Instead, ABC approaches build on the assumption that samples can be drawn from $\PP^{X,Y}$, mostly via a two-step procedure involving first drawing a sample $x$ from $\PP^X$ and then drawing $y^\prime$ from the likelihood $\PP^{Y \mid X = x}$. The drawn $x$ is then either accepted as an approximate sample of the posterior $\PP^{X \mid Y = y}$ or rejected based on some acceptance criterion  relating the drawn $y^\prime$ to $y$, which particularly in high dimensions usually needs to involve fitted summary statistics to be computationally feasible. We share the objective to sample from the posterior $\PP^{X \mid Y}$, but our situation is rather the opposite: whereas the likelihood $\PP^{Y \mid X}$ is simply a multivariate Laplacian, whence easy to implement numerically, the difficulty comes from the fact that no information about the prior $\PP^X = \alpha$ other than a finite number of samples $x_1,\ldots,x_n$ is assumed to be given. So while one may broadly view our algorithm as an ABC method, it avoids the typical rejection-based two-step sampling procedure because it is not needed in our setup.
\end{remark}

\section{Limits in high dimensions}\label{s:high}
\subsection{Convergence to the nearest neighbours in $\M$}\label{s:convergence}
One might object that the mapping $X_0 \mapsto X_\tau$ considered in the previous section does not represent the ``true'' mechanism by which the observation $y$ was generated,
and moreover depends on an unknown parameter $\sigma$. Remarkably, however, this dependence on $\sigma$ disappears in high dimensions. Even more surprisingly,
the posterior law $\PP^{X_0 \mid X_\tau = y} = \PP^{Y_\zeta \mid Y_0 = y}$ becomes increasingly concentrated around the projection of $y$ onto the support $\M$
of the prior distribution $\alpha$. In other words, the endpoint $Y_\zeta$ of the backward process approximates the point on $\M$ closest to the observed $y$.

This reveals an important geometric interpretation: the procedure effectively implements a nearest-neighbour-type projection, but not merely onto the empirical training points
$x_1, \dots, x_n$. Rather, it selects from the entire latent manifold $\M$ supporting the unperturbed data. This manifold---or an implicit estimate thereof---is
encoded in the neural network described in Remark~\ref{r:practice}.\ref{r:hier}, which is trained to approximate the backward drift $b$. This insight not only
justifies the use of diffusion-based denoising in settings with unknown noise structure, but also highlights its potential as a robust geometric
inference method in high-dimensional, label-scarce regimes.

The following theorem shows that, in high dimensions $d$, the posterior distribution \eqref{e:bayes} concentrates near the subset of $\M$ that lies closest to the observation $y$. Its assumptions and statement are illustrated below in Remark~\ref{r:sinn} and Section~\ref{su:bilder}.

\begin{theorem}[Asymptotic concentration near the projection I]\label{t:nn}
Let $\delta,\varepsilon>0$ and fix an observation $y\in\rr^d$.
If $\alpha(B(y,r))>\varepsilon$ for some ball $B(y,r)$ of radius $r>0$ centred at $y$, then
\begin{equation}\label{e:fast1}
\PP^{Y_\zeta|Y_0=y}\bigl(\M\cap B(y,(1+\delta)r)\bigr)\;\ge\; 1-\frac{1}{\varepsilon}(1+\delta)^{\,2-d}.
\end{equation}
\end{theorem}
Note that the bound and the constants involved are independent of the noise parameter $\sigma$ and uniform over all measures $\alpha$ subject only to the local mass condition $\alpha(B(x,r))\ge\varepsilon$.
\begin{remark}\label{r:sinn}
Since $\varepsilon$ and $\delta$ should be regarded as small, the sets
$\M\cap B(y,r)$ and, consequently, $\M\cap B(y,(1+\delta)r)$ contain only those points
of the support $\M$ of the clean data distribution $\alpha$ that are nearest to the observed value $y$.
Moreover, the right-hand side of \eqref{e:fast1} approaches one as $d$ increases. Hence the estimate \eqref{e:fast1} confirms our claim that, for large $d$,
the posterior law is concentrated in the region of $\M$ closest to the observed $y$.
\end{remark}

\begin{remark} 
Even though $\sigma$ does not play a rôle in Theorem \ref{t:nn}, an actual implementation of the algorithm with inaccessible data distribution $\alpha$ as described in Remark \ref{r:practice} needs to calibrate this parameter appropriately to ensure that the score $\bb(y)$ is learned accurately both in an environment of the observation $y$ (for precise guidance towards $\M$) and in an environment of $\M$ (for fine-tuning the output). Since the location of $y$ is determined by the true corruption mechanism this suggests that $\sigma$ should ideally be fitted to the noise variance. We leave experimental exploration of such effects to future work.
\end{remark}

The proof of Theorem~\ref{t:nn} relies in part on the following lemma.

\begin{lemma}[Monotone domination of the resolvent kernel]\label{lem:monotone_domination}
For $R>0$ and $y\in\rr^d$, we have
\[\frac{\displaystyle \int_{B(y,R)} G_{\sigma^2}(x,y)\,\alpha(dx)}
     {\displaystyle \int_{\rr^d} G_{\sigma^2}(x,y)\,\alpha(dx)}
\;\;\ge\;\;
\frac{\displaystyle \int_{B(y,R)} |x-y|^{2-d}\,\alpha(dx)}
     {\displaystyle \int_{\rr^d} |x-y|^{2-d}\,\alpha(dx)}.\]
\end{lemma}

\begin{proof}
Write
\[G_{\sigma^2}(x,y)=C_d(\sigma)\,|x-y|^{2-d}\,f(|x-y|)\]
with
$f(\rho)\coloneq (\kappa\rho)^{\nu}K_\nu(\kappa\rho)$,
$C_d(\sigma)\coloneq \frac{2}{\sigma^2}(2\pi)^{-d/2}>0$,
$\nu=\tfrac d2-1$, and $\kappa=\sqrt{2}/\sigma$.
By \cite[Appendix 2]{Borodin}, we have
\[\frac{d}{dz}\big(z^\nu K_\nu(z)\big)=-\,z^\nu K_{\nu-1}(z)\quad\text{ for }z>0,\nu>0,\]
which implies that $z\mapsto z^\nu K_\nu(z)$ is strictly decreasing on $(0,\infty)$. Consequently, $f$ is strictly decreasing on $(0,\infty)$.

Let $B\coloneq B(y,R)$, $B^{\C}\coloneq \rr^d\setminus B$, and set
$w(x)\coloneq |x-y|^{2-d}$ as well as $g(x)\coloneq f(|x-y|)$.
Since $f$ is decreasing, we have $g(x)\ge f(R)$ for $x\in B$ and $g(x)\le f(R)$ for $x\in B^{\C}$. Therefore,
\begin{align*}
\int_B w(x)g(x)\,\alpha(dx)\ &\ge\ f(R)\int_B w(x)\,\alpha(dx),\\
\int_{B^{\C}} w(x)g(x)\,\alpha(dx)\ &\le\ f(R)\int_{B^{\C}} w(x)\,\alpha(dx).
\end{align*}
Let $N_h\coloneq \int_B w g\,d\alpha$ and $D_h\coloneq \int w g\,d\alpha$. Then,
\[D_h = N_h+\int_{B^{\C}} w g\,d\alpha
\ \le\ N_h+f(R)\int_{B^{\C}} w\,d\alpha.\]
Hence, noting that $u\mapsto \frac{u}{u+c}$ is increasing on $\rr_+$ for any $c\ge0$,
\[\frac{N_h}{D_h}
\ \ge\ \frac{N_h}{\,N_h+f(R)\int_{B^{\C}} w\,d\alpha\,}
\ \ge\ \frac{f(R)\int_B w\,d\alpha}{\,f(R)\int_B w\,d\alpha+f(R)\int_{B^{\C}} w\,d\alpha\,}
\ =\ \frac{\int_B w\,d\alpha}{\int w\,d\alpha}.\]
Finally, the constant $C_d(\sigma)$ cancels from the ratio, which yields
\[\frac{\displaystyle \int_{B(y,R)} G_{\sigma^2}(x,y)\,\alpha(dx)}
     {\displaystyle \int G_{\sigma^2}(x,y)\,\alpha(dx)}
\ \ge\
\frac{\displaystyle \int_{B(y,R)} |x-y|^{2-d}\,\alpha(dx)}
     {\displaystyle \int |x-y|^{2-d}\,\alpha(dx)}.\qedhere\]
\end{proof}

\begin{proof}[Proof of Theorem \ref{t:nn}]
Set $R\coloneq (1+\delta)r$ and $B\coloneq B(y,R)$.
Corollary \ref{t:hitting_distribution} gives
\[\PP\!\left(Y_\zeta\in B\ \big|\ Y_0=y\right)
= \frac{\displaystyle \int_{B} G_{\sigma^2}(x,y)\,\alpha(dx)}
        {\displaystyle \int G_{\sigma^2}(x,y)\,\alpha(dx)}.\]
Applying Lemma~\ref{lem:monotone_domination} yields the lower bound
\begin{equation}\label{eq:monotone-lb}
\PP\!\left(Y_\zeta\in B\ \big|\ Y_0=y\right)
\;\ge\;
\frac{\displaystyle \int_{B} |x-y|^{2-d}\,\alpha(dx)}
     {\displaystyle \int |x-y|^{2-d}\,\alpha(dx)}.
\end{equation}
We now estimate the right-hand side of \eqref{eq:monotone-lb}. On $B(y,r)$, we have
$|x-y|\le r$, hence
\[\int_{B(y,r)} |x-y|^{2-d}\,\alpha(dx)
\;\ge\; r^{\,2-d}\,\alpha\big(B(y,r)\big)
\;\ge\; \varepsilon\,r^{\,2-d}.\]
On $B^{\C}$, we have $|x-y|\ge (1+\delta)r$, hence
\[\int_{B^{\C}} |x-y|^{2-d}\,\alpha(dx)
\;\le\; \big((1+\delta)r\big)^{2-d}.\]
Therefore,
\[\frac{\displaystyle \int_{B^{\C}} |x-y|^{2-d}\,\alpha(dx)}
     {\displaystyle \int   |x-y|^{2-d}\,\alpha(dx)}
\;\le\;
\frac{\displaystyle \int_{B^{\C}} |x-y|^{2-d}\,\alpha(dx)}
     {\displaystyle \int_{B}   |x-y|^{2-d}\,\alpha(dx)}
\;\le\; \frac{(1+\delta)^{2-d}}{\varepsilon},\]
which rearranges to
\[\frac{\displaystyle \int_{B} |x-y|^{2-d}\,\alpha(dx)}
     {\displaystyle \int |x-y|^{2-d}\,\alpha(dx)}
\;\ge\; 1-\frac{1}{\varepsilon}(1+\delta)^{2-d}.\]
Combining this with \eqref{eq:monotone-lb}
and noting that $Y_\zeta\in\M$ yields
\[\PP\!\left(Y_\zeta \in \M\cap B(y,(1+\delta)r)\ \big|\ Y_0=y\right)
\;\ge\; 1-\frac{1}{\varepsilon}(1+\delta)^{2-d}.\qedhere\]
\end{proof}

\smallskip

\noindent\textbf{Leading-order approximations}\label{su:leading_order}\quad
Theorem \ref{t:nn} shows that, with probability close to one,
the stopped backward diffusion $Y_\zeta$ yields an element of $\M$ that lies almost as close to the observation $y$ as the nearest neighbour on~$\M$. To illustrate this perhaps surprising concentration effect, we now derive first-order approximations of the relevant quantities.

We begin by recalling that
\begin{equation*}\label{e:K}
K_\nu(x)\sim \sqrt{\frac{\pi}{2\nu}}\left(\frac{2\nu}{\e}\right)^{\nu}x^{-\nu}
\e^{-{x^2\over4\nu}}
\end{equation*}
uniformly for \(0<x\le a_\nu\), where \(a_\nu=O(\sqrt{\nu})\), as \(\nu\to\infty\). Here, ``$\sim$'' indicates that the ratio tends to 1; see \cite[§10.41]{NIST:DLMF}.
Substituting this asymptotic form of the Bessel function into the definition of the resolvent kernel yields
\[G_{\sigma^2}(x,y)\approx{1\over\sqrt{2\pi\nu}}\Bigl({\nu\over\pi\e}\Bigr)^\nu{|x-y|^{-2\nu}\over\sigma^2}\exp\!\left(-{|x-y|^2\over2\nu\sigma^2}\right)\]
in the limit of large dimension~$d$, uniformly for $|x-y|=O(\sqrt d)$. Consequently, if $\operatorname{supp} \alpha$ is contained in a ball of radius $O(\sqrt{d})$ around the origin,
\begin{equation}\label{e:limitb}
\bb(y)\approx d\frac{\int \frac{x-y}{|x-y|^d}\,\exp\!\left(-{|x-y|^2\over(d-2)\sigma^2}\right)\,\alpha(dx)}{\int {|x-y|^{2-d}}\,\exp\!\left(-{|x-y|^2\over(d-2)\sigma^2}\right)\,\alpha(dx)}, \quad y = O(\sqrt{d}).
\end{equation}
As $d\to\infty$, the measures $\beta_d(dx)\coloneq |x-y|^{-d}\,\alpha(dx)$ become increasingly concentrated on the set $\argmin_{x\in\M}|x-y|$.
Thus, the backward drift $\bb(y)$ asymptotically points in the direction of the orthogonal projection of $y$ onto $\M$, at least if this projection is unique. Moreover, from~\eqref{e:limitb}, we obtain
\[|\bb(y)|\approx{d\over\inf_{x\in\M}|x-y|}.\]
for large $d$, which implies that the estimated backward drift can be used to measure the distance of an observation $y$ from
the manifold $\M$.
Finally, the right-hand sides of \eqref{e:MC} and \eqref{e:mc} simplify asymptotically to
\begin{equation*}\label{e:bmin2}
\E\biggl(\Bigl| \bb\bigl(X_0+\sigma\sqrt{U}V\bigr) +\frac{V}{\sigma\sqrt{U}}\Bigr|^2\bm{1}_{\{\sigma\sqrt{U}|V|\geq\delta\}}\biggr)
\end{equation*}
and accordingly
\begin{equation*}\label{e:mc2}
\lim_{N\to\infty}{1\over N}\sum_{i=1}^N\biggl(\biggl| \bb\bigl(\xi_i+\sigma\sqrt{u_i}v_i\bigr)
+ \frac{v_i}{\sigma\sqrt{u_i}}\biggr|^2\bm{1}_{\{\sigma\sqrt{u_i}|v_i|\geq\delta\}}\biggr).
\end{equation*}
\smallskip 

\noindent\textbf{Asymptotic concentration near the projection with estimated drift}\quad
As described in Remark \ref{r:practice}, in practice, the true drift function $b$ of $\bm{Y}$ is unknown, and
we can at best simulate a diffusion $\widetilde Y$ driven by an estimated drift $\widetilde b\approx \bb$. From an applied viewpoint, it is thus natural to ask whether a weaker form of~\eqref{e:fast1} still holds for~$\widetilde Y$.

A moment's reflection shows that we cannot hope to stop precisely
on $\M$ because the latter is not uniquely determined by its finite subset
$\{x_1,\dots,x_n\}$. Instead, we must be content with ending up in a
small neighbourhood $\widetilde\M$ of $\M$. The following theorem makes this
statement precise: Given that, for large sample size $n$,
a neural network or similar  optimisation tool provides
a sufficiently close approximation $\widetilde b_n$ of $b$,
the simulated value $\widetilde Y_{\widetilde\zeta}$ will
be satisfactory with high probability.
\textit{Satisfactory} here means that it is not far from
$\M$ and that, as in Theorem \ref{t:nn}, it is not much farther from the observation $y$ than the nearest neighbour in $\M$.
In accordance with Remark \ref{r:practice}.\ref{i:schwelle}, the stopping time $\widetilde\zeta$
is chosen to be of the form
\begin{equation}\label{e:zetaschlange}
\widetilde\zeta = \inf\Bigl\{s\geq0:
\bigl\|\widetilde b_n(\widetilde Y)\bigr\|_{L^2(s)}\geq M\Bigr\},
\end{equation}
for some sufficiently large $M\in\rr_+$.

For the following statement, we assume $\widetilde b_n$ to be an estimator
of $\bb$ that is based on an i.i.d.\ random sample $X_1,\dots,X_n$ from $\alpha$.
For fixed $n$ and $M\in\rr_+$, we consider a diffusion $\widetilde Y$ of the form
\begin{equation}\label{e:Z2}
d\widetilde Y_s = \widetilde b_n(\widetilde Y_s)\bm{1}_{\{s\leq\widetilde \zeta\}}\,ds
+ \bm{1}_{\{s\leq\widetilde \zeta\}}\,d\widetilde W_s,
\end{equation}
where $\widetilde\zeta$ is as in \eqref{e:zetaschlange} and $\widetilde W$ denotes a Brownian motion that is independent of $\widetilde{b}_n$.
\begin{theorem}[Asymptotic concentration near the projection II]\label{t:nn2}
Fix an observation $y \in \rr^d$. Suppose that
\begin{itemize}
\item for any $\widetilde\delta,\delta,\varepsilon>0$, 
we have
\begin{equation}\label{e:ngross}
\PP\biggl(\Bigl\|\bigl(\widetilde b_n(\bm{Y})-\bb(\bm{Y})\bigr)\bm{1}_{\{Y\notin\M_{\widetilde\delta}\}}\Bigr\|_{L^2(\zeta)}>\delta\,\bigg|\,Y_0=y\biggr)<\varepsilon
\end{equation}
for all sufficiently large $n$, where $\M_{\widetilde\delta}$ is defined as in Theorem \ref{t:b} and $\bm{Y}$ satisfies \eqref{e:Z} with the same driving Brownian motion as in \eqref{e:Z2};
\item
for each $n \in \nn$ and $\widetilde\delta>0$,
the function $\widetilde b_n$ is Lipschitz on
$\M_{\widetilde\delta}^{\C}$ with Lipschitz constant
$L_{\widetilde\delta}$ independent of~$n$.
\end{itemize}
Let $\delta,\varepsilon,\widetilde\delta,\widetilde\varepsilon>0$.
If $\alpha(B(y,r))>\varepsilon$ for some ball $B(y,r)$
with radius $r>0$ around $y$, then
\begin{equation}\label{e:fast2}
\PP^{\widetilde Y_{\widetilde \zeta}|\widetilde Y_0=y}
\bigl(\M_{\widetilde\delta}\cap B(y,(1+\delta)r)\bigr)\;>\;
1-\frac{1}{\varepsilon}(1+\delta)^{\,2-d}-\widetilde\varepsilon,
\end{equation}
for sufficiently large  $M\in\rr_+$ and $n\in\nn$, where
$\widetilde Y,\widetilde\zeta$ are as in (\ref{e:Z2}, \ref{e:zetaschlange}).
\end{theorem}
\begin{proof}
\textit{Step 1:}
Without loss of generality, we choose $\widetilde\delta$ sufficiently small such that
\[\frac{1}{\varepsilon}
\biggl(1+\delta-{\widetilde\delta\over r}\biggr)^{\,2-d}
-\frac{1}{\varepsilon}(1+\delta)^{\,2-d}
\leq{\widetilde\varepsilon\over10}.\]
(Otherwise, the second term on the right-hand side of \eqref{e:fast2} would read as
$\frac{1}{\varepsilon}(1+\delta-{\widetilde\delta\over r})^{\,2-d}$.)
For sufficiently large $T$, we have $\PP(\zeta\geq T)\leq\widetilde\varepsilon/10$.
Hence, we may restrict our attention to stopping before $T$ without significantly affecting the success probability.

\textit{Step 2:}
For
\[\zeta_M \coloneq  \inf\Bigl\{s\geq0:
\bigl\|\bb(\bm{Y})\bigr\|_{L^2(s)}\geq M\Bigr\},\]
we have $\zeta_M\uparrow\zeta$ and $Y_{\zeta_M}\to Y_\zeta$ as $M\to\infty$,
by Theorem \ref{t:stopping} and continuity of $\bm{Y}$.
Choose $M$ sufficiently large such that
\[\sup\bigl\{|Y_s- Y_\zeta|: s\in\auf\zeta_{M/2},\zeta\zu\bigr\}
>\widetilde\delta/2\]
occurs with probability at most $\widetilde\varepsilon/10$.

\textit{Step 3:}
Define
\[\xi_k \coloneq  \inf\Bigl\{s\geq0: d(Y_s,\M)\leq 1/k\Bigr\}.\]
Since
\[\bigcap_{k\in\nn}\bigl\{\xi_k<\zeta_{2M}\bigr\}\subset\Bigl\{\inf_{s\leq\zeta_{2M}} d(Y_s,\M)=0\Bigr\}
=\{Y_s\in\M \text{ for some }s\leq\zeta_{2M}\}
\subset\{\zeta\leq\zeta_{2M}\},
\]
which has probability 0 by Theorem \ref{t:stopping}, we can choose $k \in \nn$ large enough such that $\xi_k < \zeta_{2M}$ occurs with probability at most $\widetilde\varepsilon/10$.

\textit{Step 4:}
Let $L \coloneq L_{1/k}$ denote the Lipschitz constant of $\widetilde b_n$ on $\M_{1/k}^{\C}$ and $\widetilde{\xi}_k \coloneq \xi_k \wedge \widetilde{\zeta}$.
Choose $n$ as in \eqref{e:ngross} for $1/k, \bar\delta, \widetilde\varepsilon/10$ in place of $\widetilde\delta, \delta, \varepsilon$, where we define
\[\bar\delta\coloneq
{M\over4}\wedge\biggl({\widetilde\delta\over2}\wedge{M\over4L\sqrt{T}}\biggr)\sqrt{\e^{-L^2 T^2}\over 2T}.\]
Let $A$ denote the corresponding exceptional event in \eqref{e:ngross}, which satisfies $\PP(A) < \widetilde\varepsilon/10$.
From
\[d\bigl(Y^{\widetilde{\xi}_k}-\widetilde Y^{\widetilde{\xi}_k}\bigr)_t
=\Bigl(\bb\bigl(Y^{\widetilde{\xi}_k}_t\bigr)-\widetilde b_n\bigl(\widetilde Y^{\widetilde{\xi}_k}_t\bigr)\Bigr)\bm{1}_{\auf0,{\widetilde{\xi}_k}\zu}(t)dt,\]
it follows that
\begin{align*}
&\bigl|Y^{\widetilde{\xi}_k}-\widetilde Y^{\widetilde{\xi}_k}\bigr|^2_t\\
&\leq 2\biggl(\int_0^t\bigl|\bb\bigl(Y_s\bigr)-\widetilde b_n\bigl(Y_s\bigr)\bigr|\bm{1}_{\auf0,{\widetilde{\xi}_k}\zu}(s)ds\biggr)^2
+2\biggl(\int_0^t\bigl|\widetilde b_n\bigl(Y_s\bigr)-\widetilde b_n\bigl(\widetilde Y_s\bigr)\bigr|\bm{1}_{\auf0,{\widetilde{\xi}_k}\zu}(s)\,ds\biggr)^2\\
&\leq 2t\int_0^t\Bigl(\bigl|\bb\bigl(Y_s\bigr)-\widetilde b_n\bigl(Y_s\bigr)\bigr|^2\bm{1}_{\auf0,{\xi_k}\zu}(s)
+\bigl|\widetilde b_n\bigl(Y_s\bigr)-\widetilde b_n\bigl(\widetilde Y_s\bigr)\bigr|^2\bm{1}_{\auf0,{\widetilde{\xi}_k}\zu}(s)\Bigr)\,ds\\
&\leq 2t\bar\delta^2 + \int_0^t 2tL^2\bigl|Y^{\widetilde{\xi}_k}-\widetilde Y^{\widetilde{\xi}_k}\bigr|^2_s\,ds
\end{align*}
on $A^{\C}$. By Grönwall's inequality, we obtain
\[\bigl(\bigl(Y^{\widetilde{\xi}_k}-\widetilde Y^{\widetilde{\xi}_k}\bigr)^\star_T\bigr)^2\leq2T\bar\delta^2 \e^{L^2T^2}
\leq \biggl({\widetilde\delta\over2}\wedge{M\over4L\sqrt{T}}\biggr)^2,\]
and hence
\[\bigl(Y-\widetilde Y\bigr)^\star_{\widetilde{\xi}_k\wedge T}\leq {\widetilde\delta\over2}\wedge {M\over4L\sqrt{T}}\]
on $A^{\C}$, where we use the notation $f^\star(s)\coloneq\sup_{r\leq s}|f(r)|$ for $f\colon\rr_+\to\rr^d$.

\textit{Step 5:}
We now show that $\widetilde\zeta \wedge T \leq \zeta_{2M} \wedge T$ and hence ${\xi_k} \wedge T \geq \widetilde\zeta \wedge T$ hold on the set $\{\zeta_{2M} \leq \xi_k\} \cap A^{\C}$. On the same set we have then in particular $\widetilde{\xi}_k \wedge T = \widetilde{\zeta} \wedge T$.
Indeed, on $\{\zeta_{2M} \leq \xi_k \wedge \widetilde\zeta \wedge T\} \cap A^{\C}$, we have
\begin{align*}
\bigl\|\widetilde b_n(\widetilde{\bm{Y}})\bigr\|_{L^2(\zeta_{2M})}
&\geq \bigl\|\bb(\bm{Y})\bigr\|_{L^2(\zeta_{2M})}
-\bigl\|\bb(\bm{Y})-\widetilde b_n(\bm{Y})\bigr\|_{L^2(\zeta_{2M})}
-\bigl\|\widetilde b_n(\bm{Y})-\widetilde b_n(\widetilde{\bm{Y}})\bigr\|_{L^2(\zeta_{2M})}\\
&\geq 2M-\bar\delta-L\sqrt{T}\bigl(\bm{Y}-\widetilde{\bm{Y}}\bigr)^\star_{\zeta_{2M}}\\
&\geq 2M-{M\over2}-{M\over2}= M,
\end{align*}
which implies that
$\widetilde\zeta\leq\zeta_{2M}$.

\textit{Step 6:}
We next show that $\zeta_{M/2} \wedge T \leq \widetilde\zeta \wedge T \leq \zeta \wedge T$ holds on the set $\{\zeta_{2M} \leq \xi_k\} \cap A^{\C}$.
The second inequality follows from Step 5, since $\zeta_{2M} \wedge T \leq \zeta \wedge T$.
For the first inequality, note that Steps 3–5 yield
\begin{align*}
\bigl\|\bb(\bm{Y})\bigr\|_{L^2(\widetilde\zeta)}
&\geq \bigl\|\widetilde b_n(\widetilde{\bm{Y}})\bigr\|_{L^2(\widetilde\zeta)}
-\bigl\|\widetilde b_n(\widetilde{\bm{Y}})-\widetilde b_n(\bm{Y})\bigr\|_{L^2(\widetilde\zeta)}
-\bigl\|\widetilde b_n(\bm{Y})-\bb(\bm{Y})\bigr\|_{L^2(\widetilde\zeta)}\\
&\geq M-L\sqrt{T}\bigl(\widetilde{\bm{Y}}-\bm{Y}\bigl)^\star_{\widetilde\zeta}- \bar\delta\\
&\geq M-{M\over4}-{M\over4}= {M\over2},
\end{align*}
and hence $\zeta_{M/2}\leq\widetilde\zeta$
on the set
$\{\zeta_{2M}\leq\xi_k\}\cap A^{\C}\cap\{\widetilde\zeta\leq T\}$.

\textit{Step 7:}
Finally, the claim follows from
\begin{align}
&\PP^{\widetilde Y_{\widetilde \zeta}|\widetilde Y_0=y}
\Bigl(\bigl(\M_{\widetilde\delta}\cap B(y,(1+\delta)r)\bigr)^{\C}\Bigr)\notag\\
&\leq \PP^{\widetilde Y_{\widetilde \zeta\wedge T} \mid \widetilde Y_0=y}
\Bigl(\bigl(\M_{\widetilde\delta}\cap B(y,(1+\delta)r)\bigr)^{\C}\Bigr)
+\PP\bigl(\zeta\geq T\,\big|\,\widetilde Y_0=y\bigr)
+\PP\bigl(\widetilde\zeta\wedge T>\zeta\wedge T\,\big|\,\widetilde Y_0=y\bigr) \notag\\
&\leq \PP^{\widetilde Y_{\widetilde \zeta\wedge T}|\widetilde Y_0=y}
\Bigl(\bigl(\M_{\widetilde\delta}\cap B(y,(1+\delta)r)\bigr)^{\C}\Bigr)
+{3\widetilde\varepsilon\over10} \tag{Steps 1 and 6}\\
&\leq\PP^{Y_{\widetilde\zeta\wedge T}|Y_0=y}
\Bigl(\bigl(\M_{\widetilde\delta/2}\cap B(y,(1+\delta)r-\widetilde\delta/2)\bigr)^{\C}\Bigr)
+{5\widetilde\varepsilon\over10} \tag{Steps 4 and 5}\\
&\leq \PP^{Y_{\zeta\wedge T}|Y_0=y}
\Bigl(\bigl(\M\cap B(y,(1+\delta)r-\widetilde\delta)\bigr)^{\C}\Bigr)
+{8\widetilde\varepsilon\over10} \tag{Steps 2 and 6}\\
&\leq \PP^{Y_{\zeta}|Y_0=y}
\Bigl(\bigl(\M\cap B(y,(1+\delta)r-\widetilde\delta)\bigr)^{\C}\Bigr)
+{9\widetilde\varepsilon\over10} \tag{Step 1}\\
&\leq\frac{1}{\varepsilon}
\bigl(1+\delta-\widetilde\delta/r\bigr)^{\,2-d}
+{9\widetilde\varepsilon\over10} \tag{Theorem \ref{t:nn}}\\
&\leq\frac{1}{\varepsilon}(1+\delta)^{\,2-d}+\widetilde\varepsilon. \tag{Step 1}
\end{align}
\end{proof}

\begin{remark}[Motivation of condition \eqref{e:ngross}]
By Theorem \ref{t:b}, its proof, and Remark \ref{r:practice}.\ref{i:schwelle}, we can hope that
a reasonable statistical learning procedure provides an estimate
$\widetilde b_n$ that is close to $\bb$ in an $L^2$-sense. More specifically, for any fixed $\widetilde\delta>0$,
\[\E\Bigl(\bigl|\widetilde b_n(X_\tau)-\bb(X_\tau)\bigr|^2
\bm{1}_{\{X_\tau\notin\M_{\widetilde\delta}\}}\Bigr)\]
is small for large $n$.
Since
\begin{flalign*}
\E\Bigl(\bigl|\widetilde b_n(X_\tau)&-\bb(X_\tau)\bigr|^2
\bm{1}_{\{X_\tau\notin\M_{\widetilde\delta}\}}\Bigr)
={1\over\sigma^2}\E\biggl(\int_0^\tau\bigl|\widetilde b_n(X_t)-\bb(X_t)\bigr|^2
\bm{1}_{\{X_t\notin\M_{\widetilde\delta}\}}\,dt\biggr)\\
&={1\over\sigma^2}\E\biggl(\int_0^\zeta\bigl|\widetilde b_n(Y_s)-\bb(Y_s)\bigr|^2
\bm{1}_{\{Y_s\notin\M_{\widetilde\delta}\}}\, ds\biggr)\\
&\geq{1\over\sigma^2}\delta^2 \int\PP\biggl(\Bigl\|\bigl(\widetilde b_n(\bm{Y})-\bb(\bm{Y})\bigl)\bm{1}_{\{\bm{Y}\notin\M_{\widetilde\delta}\}}\Bigr\|_{L^2(\zeta)}>\delta \,\bigg| \,
Y_0 = y\biggr)\, \PP^{X_\tau}(dy)
\end{flalign*}
holds by Lemma \ref{l:exp},
this motivates \eqref{e:ngross} as a reasonable requirement of
closeness of $\widetilde b_n$ and $\bb$.
\end{remark}

\begin{example}[Learning projections onto arbitrary given sets] \label{ex:proj}
Our approach can be used as a numerical method for computing the
projection onto a given polar set $\M\subset\rr^d$ for large $d$,
provided that we know how to draw i.i.d.\ samples from
some law $\alpha$ with support $\M$.
By generating training data via the parametric model \eqref{e:noidea} and training a diffusion model accordingly,
the denoising procedure yields---in high dimensions---an approximate projection of any input $Y$ onto $\M$.

The ``blessing of dimensionality'' aspect of the algorithm is particularly relevant here, since even for relatively simple geometric objects
the numerical approximation of the orthogonal projection becomes prohibitively expensive in high dimensions. 
In the context of \textit{reflected} diffusion models,
that is, generative models that operate with reflected diffusions as noising and denoising processes \cite{lou23, fishman23, holk25},
various implementations of the reflection step have been proposed; see, e.g., \cite{fishman24}.
However, replacing the reflection by a projection step, as theoretically supported by \cite{slominski94}, has been ruled out for the computational reasons mentioned above. 
Our proposed stochastic algorithm may therefore provide a new simulation perspective in this setting, offering a general flexible framework for numerical geometric reconstruction
in high dimensions that applies beyond classical structures such as sparsity or manifold structure.
\end{example}

\subsection{Application to image processing}\label{su:bilder}
In this short section, we illustrate Theorems~\ref{t:nn} and \ref{t:nn2} in a canonical statistical setting arising in inverse problems and image reconstruction.

To this end, let $\alpha$ be a probability measure on $L^2((0,1]^2)$, referring to the law of images of interest.
Given an image $x \in L^2((0,1]^2)$, its discretised version at resolution $d = (2^k)^2$ is defined by smoothing over pixel blocks:
\[\Phi_d(x)_{i,j} \coloneq  2^{2k} \int_{(i-1)/2^k}^{i/2^k} \int_{(j-1)/2^k}^{j/2^k} x(u,v) \, du \, dv, \quad i,j = 1, \dots, 2^k,\]
so that $\Phi_d(x) \in \mathbb{R}^{d}$. The pushforward measure $\alpha_d \coloneq  \alpha^{\Phi_d}$ then describes the law of discretised images at resolution $d$.

Fix $y \in L^2((0,1]^2)$, interpreted as a noisy observation. For given $\delta, \varepsilon > 0$, choose $r > 0$ such that
\[\alpha\big(B(y, r)\big) > \varepsilon.\]
Define $r_d \coloneq  \sqrt{d} \, r$, and set $y_d \coloneq  \Phi_d(y) \in \rr^d$. Then, for any $z \in L^2((0,1]^2)$, Jensen's inequality implies that
\[\big\lvert \Phi_d(z) - \Phi_d(y) \big\rvert \leq \sqrt{d} \lVert z -y \rVert,\]
where $\|z\|\coloneq \sqrt{\int_0^1\int_0^1z(u,v)^2dudv}$ denotes the usual $L^2$-norm.
Therefore,
\[B(y,r) \subset (\Phi_d)^{-1}\big(B(\Phi_d(y), \sqrt{d}r )\big) = (\Phi_d)^{-1}\big(B(y_d, r_d) \big).\]
By our choice of $r,\varepsilon > 0$, this yields
\[\alpha_d\big(B(y_d,r_d)\big) = \alpha\big((\Phi_d)^{-1}(B(y_d,r_d) \big) \geq \alpha\big(B(y,r) \big) > \varepsilon.\]
Hence, $\alpha_d,y_d,r_d$ satisfy the assumptions of Theorem~\ref{t:nn} for $\varepsilon$, which thus gives
\[\PP^{Y_\zeta \mid Y_0 = y_d}\big(\mathscr{M}_d \cap B(y_d,(1+\delta)r_d)\big) \geq 1 - \frac{1}{\varepsilon}(1+\delta)^{2-d},\]
where $\mathscr{M}_d$ denotes the support of $\alpha_d$.
As is shown below, the pushforward property implies
\begin{equation}\label{e:bilder}
\lim_{d \to \infty} \alpha_d\bigl(B(y_d, (1+\delta)r_d)\bigr) =
\alpha\bigl(B(y, (1+\delta)r)\bigr),
\end{equation}
so that the preimage under $\Phi_d$ of the Euclidean ball $B(y_d, (1+\delta)r_d) \subset \rr^d$ corresponds to a shrinking neighbourhood of $y$ in the $L^2$-norm.
In particular, $\alpha_d(B(y_d, (1+\delta)r_d))$ can become arbitrarily small even though the posterior concentrates in this set.

To verify \eqref{e:bilder}, consider the probability space
$((0,1]^2,\mathscr B((0,1]^2),\lambda|_{(0,1]^2})$. For $k\in\nn$, let $\mathscr F_k$ denote the $\sigma$-field on $(0,1]^2$ which is
generated by the squares $({(i-1)/2^k},{i/2^k}]\times({(j-1)/2^k},{j/2^k}]$ with $i,j=1,\dots,2^k$.
For $z=(z_{i,j})_{i,j=1,\dots,2^k}\in\rr^{2^{2k}}$, define a corresponding $\widetilde z\in L^2((0,1]^2)$ via
\[\widetilde z(u,v)\coloneq \sum_{i,j=1,\dots,2^k} z_{i,j}
\bm{1}_{({i-1\over 2^k},{i\over 2^k}]\times({j-1\over 2^k},{j\over 2^k}]}(u,v),
\quad (u,v)\in(0,1]^2.\]
For fixed $x\in L^2((0,1]^2)$, it is easy to verify that
$\E(x \mid \mathscr F_k)=\widetilde\Phi_{2^{2k}}(x)$ and $\widetilde{\Phi}_{2^{2k}}(\widetilde{\Phi}_{2^{2(k+1)}}(x)) = \widetilde\Phi_{2^{2k}}(x)$,
which implies that $(\widetilde\Phi_{2^{2k}}(x))_{k\in\nn}$ is a martingale relative to the filtration $(\mathscr F_k)_{k\in\nn}$.
By Doob's martingale convergence theorem, it follows that $\widetilde\Phi_{2^{2k}}(x)\to x$ in $L^2((0,1]^2)$ in the sense that
\[\bigl\|\widetilde\Phi_{2^{2k}}(x)-x\bigr\|\to0\ \quad \text{as } k\to\infty.\]
Now we switch to the probability space $(L^2((0,1]^2),\mathscr B(L^2((0,1]^2)),\alpha^y)$,
where $\alpha^y(B) \coloneq\alpha(B+y)$ for measurable $B\subset L^2((0,1]^2)$.
Since pointwise convergence implies convergence in law, the Portmanteau theorem yields
\[\int \bm{1}_B\bigl(\widetilde\Phi_{2^{2k}}(z)\bigr)\,\alpha^y(dz)\stackrel{k\to\infty}{\longrightarrow} \int \bm{1}_B(z)\,\alpha^y(dz)\]
for all continuity sets $B$ of $\alpha^y$.
From
\[\alpha_d\bigl(B(y_d, (1+\delta)r_d)\bigr)
=\int \bm{1}_{B(0,(1+\delta)r)}\bigl(\widetilde\Phi_{2^{2k}}(z)\bigr)\,\alpha^y(dz)\]
and
\[\alpha\bigl(B(y, (1+\delta)r)\bigr)
=\int \bm{1}_{B(0,(1+\delta)r)}(z)\,\alpha^y(dz),\]
we conclude that \eqref{e:bilder} holds provided that
$\alpha^y(\partial B(y,(1+\delta)r))=0$.
This, however, is true for all but at most countably many $\delta>0$.

\begin{appendix}
\section{Technical tools}
\begin{lemma}\label{l:cond}
Let $X$ be a random variable with values in some measurable space $(A,\mathscr A)$.
Moreover, let $Y$ be a real-valued random variable such that
$\E(|Y| \mid X)<\infty$ almost surely.
We denote by $f^\star\colon A\to\rr$ the factorisation of the conditional expectation
of $Y$ given $X$, i.e.\
\[f^\star(x) \coloneq \E(Y \mid X=x),\quad x\in A.\]
For some measurable set $B\in\mathscr A$, let $C$ be an event such that
$X^{-1}(B)\subset C$ and $\E(Y^2\bm{1}_C)<\infty$.
Then, $f^\star\bm{1}_B=\widetilde f\bm{1}_B$ holds $\PP^X$-a.s.\ for the (or, more precisely, any)
minimiser $\widetilde f$ of
\[f\mapsto \E\bigl((f(X)-Y)^2\bm{1}_C\bigr)\]
over all measurable functions $f\colon A\to\rr$.
\end{lemma}
\begin{proof}
If $f(X) \bm{1}_C \notin L^2(\PP)$, we have $\E((f(X)-Y)^2\bm{1}_C) = \infty$ because $Y\bm{1}_C \in L^2(\PP)$.
It therefore suffices to only consider measurable $f\colon A \to \rr$ such that $f(X) \bm{1}_C \in L^2(\PP)$.

\textit{Step 1:}
If $C=X^{-1}(B)$, the claim follows because
\begin{align*}
&\E\bigl((f(X)-Y))^2\bm{1}_C\bigr)\\
&=\E\bigl((f(X)-f^\star(X)+f^\star(X)-Y))^2\bm{1}_C\bigr)\\
&=\E\bigl((f(X)-f^\star(X))^2\bm{1}_C\bigr)
+2\E\bigl(\bigl(f(X)-f^\star(X)\bigr)\bigl(f^\star(X)-\E(Y \mid X)\bigr)\bm{1}_C\bigr)\\
&\qquad+\E\bigl((f^\star(X)-Y))^2\bm{1}_C\bigr)\\
&=\E\bigl((f(X)-f^\star(X))^2\bm{1}_C\bigr)+0
+\E\bigl((f^\star(X)-Y))^2\bm{1}_C\bigr)\\
&\geq \E\bigl((f^\star(X)-Y))^2\bm{1}_C\bigr)
\end{align*}
for any $f$ as above.

\textit{Step 2:}
For general $C$, we have
\[
\E\bigl((f(X)-Y))^2\bm{1}_C\bigr)=\E\Bigl(\bigl(f(X)\bm{1}_B(X)-Y)\bigr)^2\bm{1}_{X^{-1}(B)}\Bigr)
+\E\Bigl(\bigl(f(X)\bm{1}_{B^{\C}}(X)-Y)\bigr)^2\bm{1}_{C\setminus X^{-1}(B)}\Bigr).
\]
The first term on the right is minimised by $f^\star$, as is shown in Step 1.
The second term depends only on $f$ on $B^{\C}$, which can be chosen and hence
minimised independently of $f$ on the set $B$.
\end{proof}

We recall a well-known fact on expected values at independent
exponential times:
\begin{lemma}\label{l:exp}
Suppose that $\bm{X}$ is a semimartingale (or, more generally, product-measurable as a mapping $(\omega,t)\to X_t(\omega)$ on $\Omega\times\rr_+\to\rr$)
and $\tau$ an independent stopping time
which is exponentially distributed with parameter $\lambda>0$. Then,
\[\E(X_\tau)=\lambda\E\biggl(\int_0^\tau X_tdt\biggr),\]
provided that the integral on either side exists.
\end{lemma}
\begin{proof}
The claim follows from
\begin{align*}
\E(X_\tau)
&=\E\bigl(\E(X_\tau|X)\bigr)\\
&=\E\biggl(\int_0^\infty X_t \lambda \e^{-\lambda t}\,dt\biggr)\\
&=\int_0^\infty \E(X_t) \lambda \e^{-\lambda t}\,dt\\
&= \lambda \int_0^\infty \E(X_t) \PP(\tau > t)\,dt\\ 
&= \lambda \int_0^\infty \E(X_t \bm{1}_{\{\tau > t\}})\,dt\\
&= \lambda \E\biggl(\int_0^\tau X_t \, dt \biggr),
\end{align*}
where we used independence of $\bm{X}$ and $\tau$ and Fubini twice.
\end{proof}
\end{appendix}

\paragraph{Acknowledgments}
Thanks are due to Maximilian Klein for correcting errors in a preliminary version of the paper. Thanks are also due to two anonymous reviewers whose thoughtful comments improved the presentation of our results.

\printbibliography

\end{document}